\documentclass[12pt]{amsart}
\usepackage{graphicx} 
\usepackage{soul} 
\usepackage{amscd,amssymb,mathtools}
\usepackage[arrow,matrix,graph,frame,poly,arc,tips]{xy}
\usepackage[dvipsnames]{xcolor}
\usepackage{graphicx}
\usepackage{calrsfs}
\usepackage[labelfont=rm]{subcaption}
\usepackage{tikz-cd} 
\usepackage{tikz}
\usetikzlibrary{decorations.markings}
\usetikzlibrary{shapes}
\usetikzlibrary{backgrounds}
\usetikzlibrary{calc}
\usetikzlibrary{shapes.misc, positioning}
\usepackage{mdwlist}
\usepackage{xfrac}
\DeclareCaptionSubType*{figure}

\usepackage{multirow}
\usepackage{enumerate}
\usepackage{verbatim}
\usepackage{comment}
\usepackage{todonotes}

  \usepackage{hyperref}

    \newcommand\ba{\begin{align*}}
    \newcommand\ea{\end{align*}}
    \newcommand\be{\begin{enumerate}}
    \newcommand\ee{\end{enumerate}}
    \newcommand\bpf{\begin{proof}}
    \newcommand\epf{\end{proof}}
    \newcommand\bpp{\begin{prop}}
    \newcommand\epp{\end{prop}}
    \newcommand\bpb{\begin{prob}}
    \newcommand\epb{\end{prob}}
    \newcommand\bd{\begin{defn}}
    \newcommand\ed{\end{defn}}
    \newcommand\bh{\begin{hint}}
    \newcommand\eh{\end{hint}}

    \newcommand\Z{\mathbb{Z}}

\DeclareMathOperator\St{st}

    \newcommand\lk{\operatorname{Lk}}

    \newcommand\gam{\Gamma}

\def\thetitle{Right-angled Artin groups of large girth and finite volume hyperbolic $3$--manifold groups}

     \theoremstyle{plain}
    
    \newtheorem{thm}{Theorem}[section]
    
    \newtheorem{lemma}[thm]{Lemma}
    \newtheorem{cor}[thm]{Corollary}
    \newtheorem{prop}[thm]{Proposition}

    \newtheorem*{claim*}{Claim}

    \theoremstyle{remark}

    \theoremstyle{definition}
    \newtheorem{defn}[thm]{Definition}
    \newtheorem{prob}{Problem}[section]

\begin{document}
\title\thetitle
    \date{\today}


    \author[T. Koberda]{Thomas Koberda}
    \address{Department of Mathematics, University of Virginia, Charlottesville, VA 22904-4137, USA}
    \email{thomas.koberda@gmail.com}
    \urladdr{https://sites.google.com/view/koberdat}

\begin{abstract}
   Let $\gam$ be a finite simplicial graph of girth at least five. In this short note, we give a proof that if $M$ is a finite volume hyperbolic $3$--manifold, then the right-angled Artin group $A(\gam)$ cannot contain $\pi_1(M)$ as a subgroup; the argument is elementary, modulo the resolution of the Virtual Fibering Conjecture and a splitting theorem due to Belegradek. In particular, if $C_n$ denotes the $n$--cycle then $A(C_n)$ cannot contain a finite volume hyperbolic $3$--manifold group for any $n\geq 3$, thus answering a question of A.~Reid.
\end{abstract}

 \maketitle

\section{Introduction}
Let $\gam$ be a finite simplicial graph, and write $A(\gam)$ for the right-angled Artin group; $A(\gam)$ is generated by the vertices of $\gam$, and two vertices commute precisely when they are adjacent in $\gam$; see~\cite{Charney,Koberda2013,Koberda2022}. The \emph{girth} of $\gam$ is the length of the shortest (full) cycle in $\gam$.

The resolution of Thurston's virtual fibering conjecture~\cite{Agol2008,Agol2013} shows, among other things, that if $M$ is a finite volume hyperbolic $3$--manifold then $G=\pi_1(M)$ admits a finite index subgroup $G_0\leq G$ and an injective homomorphism to some right-angled Artin group. An orthogonal question asks whether a particular right-angled Artin group can contain a finite volume hyperbolic $3$--manifold group, and in general this seems to be a difficult question; see~\cite{RushtonMO}. In fact, there does not seem to be a complete characterization of right-angled Artin groups which contain closed hyperbolic surface subgroups, which is the corresponding question in dimension two.

In this note, we will establish the following, which does not seem to have appeared in the literature:

\begin{thm}\label{thm:main}
    Let $\gam$ be a graph of girth at least five and let $M$ be a finite volume hyperbolic $3$--manifold. Then $A(\gam)$ does not contain a copy of $\pi_1(M)$ as a subgroup.
\end{thm}

Ruling out hyperbolic $3$--manifold subgroups of $\Z^3$ and $F_2\times F_2$ is straightforward, and so one obtains the following resolution of a question due to Alan Reid:

\begin{cor}\label{cor:cycle}
    Let $C_n$ denote an $n$--cycle. Then $A(C_n)$ does not contain the fundamental group of a finite volume hyperbolic $3$--manifold.
\end{cor}

A classical result of Droms--Servatius--Servatius implies that $A(C_n)$ contains closed hyperbolic surface groups whenever $n\geq 5$~\cite{SDS1989}, and so Corollary~\ref{cor:cycle} complements their result for higher dimensional manifolds.

Hyperbolicity is used in an essential way in Theorem~\ref{thm:main}, since certain right-angled Artin groups, the components of whose underlying graphs are allowed to be trees and triangles, are in fact (virtually) identified with fundamental groups of $3$--manifolds; see~\cite{Droms87,Gordon}.

\section{Reduction to connected graphs and ruling out short cycles}

From now on, a hyperbolic $3$--manifold will always mean a finite volume hyperbolic $3$--manifold. We will appeal to standard facts about $G=\pi_1(M)$ without comment; see~\cite{Thurston}, for example. We will use the fact
that $M$ is virtually fibered over the circle, by a theorem of Agol~\cite{Agol2013}:
\begin{thm}\label{thm:agol}
    The group $G$ admits a finite index subgroup $H\leq G$ and a surjective homomorphism $\chi\colon H\longrightarrow\Z$ such that $\ker\chi$ is finitely generated.
\end{thm}

The kernel $F=\ker\chi$ in Theorem~\ref{thm:agol} is either a closed hyperbolic surface group when $M$ is closed, and a finitely generated free group otherwise.

Since we are interested in proving that hyperbolic $3$--manifold groups cannot lie inside of right-angled Artin groups on graphs of large girth, we may freely pass to finite index subgroups of a hyperbolic $3$--manifold group $G$, and hence we may assume without loss of generality that $G$ satisfies the conclusion of Theorem~\ref{thm:agol} already. For the same reason, we may assume that $M$ is orientable and has torus cusps.

It is convenient to assume that $\gam$ is a connected graph. It is trivial to see that if $\gam$ is disconnected then $\gam$ splits as the free product of the right-angled Artin groups on the components of $\gam$. By the Kurosh Subgroup Theorem, a subgroup of a free product is either conjugate into a free factor, is cyclic, or itself splits as a nontrivial free product. For hyperbolic $3$--manifolds, this is not possible; we have the following result due to Belegradek~\cite{Belegradek2002,Belegradek2007}:

\begin{thm}\label{lem:belegradek}
    A hyperbolic $3$--manifold group $G$ does not admit a nontrivial splitting over an elementary subgroup.
\end{thm}

Here, an \emph{elementary subgroup} is one which has a finite orbit on the boundary of hyperbolic space, which in the case of hyperbolic $3$--manifolds implies that it must be virtually abelian.
In particular, $G$ cannot split as a nontrivial free product, nor (as will be important later) can it split as an amalgamated product or HNN extension with amalgamating subgroup isomorphic to $\Z$.

We may thus always assume, for the purposes of Theorem~\ref{thm:main}, that $\gam$ is connected. For the purposes of Corollary~\ref{cor:cycle}, we have:

\begin{lemma}\label{lem:short-cycle}
    A hyperbolic $3$--manifold group $G$ does not embed as a subgroup of $\Z^3=A(C_3)$, nor $F_2\times F_2=A(C_4)$.
\end{lemma}
\begin{proof}
    The first claim is trivial. For the second, each nontrivial element of $G$ has abelian centralizer, and $G$ itself has no abelian normal subgroup. It is straightforward to argue then that projecting $G\leq F_2\times F_2$ onto one of the two factors, it must be the case that the projection map is injective when restricted to $G$, whence $G$ is free or cyclic, a contradiction.
\end{proof}

\section{Right-angled Artin groups on graphs of large girth}

We now establish Theorem~\ref{thm:main}, with Corollary~\ref{cor:cycle} following from it and Lemma~\ref{lem:short-cycle}. We always assume that $\gam$ has girth at least five and is connected. Note that if $J\subseteq\gam$ is a (full, nontrivial) join subgraph of $\gam$ then $J$ is contained in the star $\St(v)$ of a vertex $v$ of $\gam$. This is a straightforward consequence of the girth assumption.

The following is an easy consequence of Servatius' Centralizer Theorem~\cite{Servatius1989} and the structure of stars in $\gam$.

\begin{lemma}\label{lem:vertex}
    Let $1\neq x\in A(\gam)$ have noncyclic centralizer. Then $x$ is conjugate into $A(\St(v))$ for some vertex $v$ of $\gam$. Moreover, if $\Z^2\cong H\leq A(\gam)$, then there exists a $g\in A(\gam)$ and a vertex $v$ of $\gam$ such that $g^{-1}Hg\cap\langle v\rangle\neq 1$.
\end{lemma}

Note that if $\pi_1(M)=G\leq A(\gam)$ then $M$ cannot be closed. Indeed, otherwise $\pi_1(M)$ has cohomological dimension three, but the cohomological dimension of $A(\gam)$ is at most two.

We fix a maximal parabolic subgroup $P\leq G=\pi_1(M)$. This group $P$ is isomorphic to $\Z^2$, and is equal to its own centralizer in $G$. In light of Lemma~\ref{lem:vertex}, $P$ picks out (not necessarily canonically) a Bass--Serre splitting of $A(\gam)$. We spell this out in more detail.

Without loss of generality, the manifold $M$ fibers over the circle, $G=\pi_1(M)$ surjects to $\Z$ via a homomorphism $\chi$, and $F=\ker\chi$ is a finitely generated nonabelian free group. Since $P\cong\Z^2$, we have that $\chi$ is nontrivial when restricted to $P$, and $C=\ker\chi\cap P$ is a nontrivial cyclic group.

Identifying $P$ with its image in $A(\gam)$ and conjugating the whole of the image of $G$ if necessary, we have that $P$ is a subgroup of $B=A(\St(v))$ for some vertex $v$ of $\gam$,
and that $P\cap\langle v\rangle\neq 1$; this follows from Lemma~\ref{lem:vertex}. Write $v^r$ for the smallest nonzero power of $v$ that lies in $P$; it is a standard consequence of Servatius' Centralizer Theorem that $B$ is the full centralizer of $v^r$ in $A(\gam)$. By the maximality of $P$ in $G$ and the fact that $P$ is equal to its own centralizer in $G$, it follows that $P$ is precisely the centralizer of $v^r$ in $G$; in particular, $F\cap B=C\cong\Z$.

Deleting $v$ from $\gam$ now induces a splitting of $A(\gam)$. Write $\pi=A(\gam\setminus\{v\})$ and $\lambda=A(\lk(v))$ (where here $\lk(v)$ denotes the link of $v$), so that $A(\gam)\cong \pi*_{\lambda} B$. Note that $\lambda$ is free (possibly cyclic) and that $B\cong\lambda\times\Z$.
Write $T$ for the corresponding Bass--Serre tree for this splitting~\cite{SerreTrees,WiltonTrees}. Write $b$ for the distinguished vertex of $T$ corresponding to $B$.

Observe that since the stabilizer of any edge of $T$ incident to $b$ is a conjugate of $\lambda$ in $B$ and because $\lambda$ is normal in $B$ and contains no nonzero power of $v$, we have that no nonzero power of $v$ stabilizes an edge incident to $b$. If $v$ had any fixed vertices other than $b$ then $v$ would stabilize the corresponding geodesic segment in $T$, thus the edge incident to $b$ on that segment. Since the action of $A(\gam)$ on $T$ is without inversions, no nontrivial element can fix a point on an edge of $T$ without also fixing a vertex. Thus:

\begin{lemma}\label{lem:v-stab}
    The vertex $b$ of $T$ is the unique fixed point of any nonzero power of $v$.
\end{lemma}

Let $T_0$ denote the minimal $F$--invariant subtree. It is a standard fact in Bass--Serre theory that, because $F$ is finitely generated, such a tree exists and is unique~\cite{WiltonTrees}.

\begin{lemma}\label{lem:minimal subtree}
    The tree $T_0$ contains the vertex $b$, admits a nontrivial action by $F$, and has a finite quotient under the action of $F$.
\end{lemma}
\begin{proof}
    Suppose first that $F$ is elliptic and so fixes a nonempty set of vertices $X$ of $T$. Since $v^r$ normalizes $F$, we have that $v^r$ fixes $X$. Taking the shortest geodesic from $b$ to $X$, we obtain a unique point $z\in X$ which must therefore be fixed by $v^r$. Lemma~\ref{lem:v-stab} implies $z=b$ and so $F$ stabilizes $b$. Since $F\cap B=C\cong \Z$, this violates the fact that $F$ is nonabelian.

    Thus, $T_0$ admits a nontrivial action of $F$. Again since $F$ is normalized by $v^r$, we can consider the shortest geodesic from $b$ to $T_0$ and conclude that $b$ is a vertex of $T_0$.

    That the quotient of $T_0$ by $F$ is finite is a standard result; choosing a finite generating set $S$ for $F$ and an arbitrary vertex $x$ in $T_0$, the subtree of $T_0$ spanned by the translates of $x$ by $S$ is compact and connected and contains a fundamental domain for an $F$--invariant tree, which must therefore contain $T_0$.
\end{proof}

We write $Y$ for the quotient of $T_0$ by $F$. Let $u\in P$ be an arbitrary element with $\chi(u)\neq 0$. Since $u$ normalizes $F$, it must preserve $T_0$ and induce an automorphism of $Y$. Choosing an arbitrary edge $e$ of $T_0$ that is incident to $b$, we have that some positive power of $u$ fixes the $F$--orbit of $e$. In particular, $u^m\cdot e=f\cdot e$ for some $f\in F$. Note that since $u\in P$, we have that $u^m\cdot e$ is incident to $b$. It follows then that $f\cdot b$ is also incident to $b$ and so $f$ stabilizes $b$; here we are using the vertex--type preservation of the action on the Bass--Serre tree. In particular, $f\in C$.

Writing $h=f^{-1}u^m$, we have that $h\in P$ stabilizes $e$. Since $f$ lies in the kernel of $\chi$, we have $\chi(h)\neq 0$ and so $h\notin F$. We obtain the following:

\begin{lemma}\label{lem:edge-stab}
    The stabilizer of $e$ in $F$ is trivial.
\end{lemma}
\begin{proof}
    Write $F_0$ for the stabilizer of $e$. Since $F_0$ stabilizes $b$, we must have $F_0\leq P$ and so $F_0\leq C$.
    
    The element $h$ and $F_0$ together lie in $P$, so $h$ must commute with $F_0$. By the observation preceding Lemma~\ref{lem:v-stab}, since both $F_0$ and $h$ stabilize $e$ and since $e$ is incident to $b$, both $F_0$ and $h$ therefore lie in the subgroup $\lambda$, which is a free group. Nontrivial abelian subgroups of free groups are cyclic, and so any nontrivial element of $F_0$ must share a common power with $h$. Since $F_0$ is in the kernel of $\chi$ and $h$ is not, this is a contradiction.
\end{proof}

Now, write $H=\langle u^m,F\rangle$. We have that $H$ has finite index in $G$ and is therefore a finite volume hyperbolic $3$--manifold group. Moreover, $H/F\cong\Z$.

Write $E$ for the full $F$--orbit of $e$. Since $u^m\cdot e=f\cdot e$ from before and because $F$ is normal in $H$, we have that $H$ leaves $E$ invariant.

Consider the forest $\mathcal{F}$ spanned by the edges that do not lie in $E$. Clearly $\mathcal F$ is also $H$--invariant. Collapsing each component of $\mathcal F$ to a vertex, we obtain an $H$--equivariant homotopy equivalence of graphs $T_0\longrightarrow \tau$, so that $\tau$ is also a tree. Note that the edges of $\tau$ are in canonical bijection with the elements of $E$, and that $F$ admits a nontrivial action on $\tau$. Moreover, $\tau$ is $F$--minimal, since any proper $F$--invariant sub-tree would pull back to a proper $F$--invariant subtree of $T_0$; finally, it is clear that $F$ acts on $\tau$ without inversions.

\begin{proof}[Proof of Theorem~\ref{thm:main}]
    By Belegradek's Theorem (see Theorem~\ref{lem:belegradek}) and because $H$ acts transitively on the edges of $\tau$, it suffices to show that the $H$--stabilizer of an edge of $\tau$ is a cyclic group, since this would furnish a splitting of a hyperbolic manifold group over $\Z$.

    Let $e\in E$ be identified with an edge of $\tau$. The stabilizer $H_0\leq H$ of $e$ agrees with the $H$--stabilizer of $e$ viewed as an edge of $T_0$. Since no nontrivial element of $F$ stabilizes $e$ by Lemma~\ref{lem:edge-stab}, it follows that $H_0$ is mapped injectively into $\Z$ by $\chi$. Since the element $h$ does not lie in the kernel of $\chi$ and stabilizes $e$, we obtain that $H_0\cong\Z$, as required.
\end{proof}

\section*{AI disclosure}
The author discussed Reid's question with ChatGPT, and the idea to use Bass--Serre theory to find an elementary splitting of a hyperbolic $3$--manifold subgroup of $A(C_5)$ and to use Belegradek's splitting theorem arose from this discussion. The arguments and writing in this note are the author's.

\section*{Acknowledgements}
The author thanks S.-h.~Kim for helpful discussions. The author is partially
supported by NSF grant DMS-2349814.

\end{document}